\documentclass{article}
\usepackage{amsthm, amsmath, amsfonts, hyperref, verbatim}
\theoremstyle{plain}
\newtheorem{thm}{Theorem}
\newtheorem{lem}[thm]{Lemma}
\newtheorem{cor}[thm]{Corollary}
\newtheorem{prop}[thm]{Proposition}
\newtheorem{axiom}[thm]{Axiom}

\theoremstyle{definition}

\newtheorem{defn}[thm]{Definition}

\DeclareMathOperator{\Coin}{Coin}
\DeclareMathOperator{\sign}{sign}
\DeclareMathOperator{\ind}{ind}
\DeclareMathOperator{\id}{id}
\DeclareMathOperator{\RT}{RT}
\DeclareMathOperator{\Aut}{Aut}
\DeclareMathOperator{\dom}{dom}

\newcommand{\R}{\mathbb{R}}
\newcommand{\Z}{\mathbb{Z}}
\newcommand{\C}{\mathbb{C}}
\newcommand{\adm}{\mathcal C}
\newcommand{\admm}{\mathcal D}
\newcommand{\Reid}{\mathcal R}
\newcommand{\lift}{\widetilde}
\newcommand{\htp}{\simeq}

\newcommand{\OO}{\mathbb{O}}

\title{Axioms for the coincidence index of maps between manifolds of the same
dimension }
\author{Daciberg Lima Gon\c{c}alves and P. Christopher Staecker}

\begin{document}
\maketitle

\begin{abstract}
We study the coincidence theory of maps between two manifolds  
of the same dimension from an axiomatic viewpoint. First we  
look at coincidences
of maps between manifolds where one of the maps is orientation true,
and give a set of axioms such that characterizes the local index (which is  
an integer valued function). Then we consider coincidence theory for  
arbitrary pairs of maps between two manifolds. Similarly we provide a  
set of axioms which
characterize the local index,  which in this case is a function with  
values in $\Z\oplus \Z_2$. We also show in each setting that the group of values for the index (either $\Z$ or $\Z\oplus \Z_2$) is determined by the axioms. 

Finally, for the general case of coincidence theory  
for arbitrary pairs of maps between two manifolds we provide a set of  
axioms which charaterize the
local Reidemeister trace which is an element of an abelian group which  
depends on the pair of functions. These results extend known results  
for coincidences between orientable differentiable manifolds.
\end{abstract}

\section{Introduction}
For two mappings $f,g:M\to N$, we say that $x\in M$ is a coincidence point of $f$ and $g$ when $f(x) = g(x)$. In 1955, Schirmer (\cite{sch55}) defined a local coincidence index in the setting where $M$ and $N$ are orientable manifolds of the same (finite) dimension. This coincidence index generalizes the well-known fixed point index, and functions as an algebraic multiplicity count for coincidence points. The index is integer valued, is invariant under homotopies of $f$ and $g$, is additive on disjoint subsets, and is nonzero when $f$ and $g$ have a coincidence which cannot be removed by homotopy.

 The characterization of this type of functions has a long story.   C.~Watts in~\cite{wa62} characterized the Euler characteristic in an axiomatic way,
using a very simple set of axioms. Using similar types of axioms, recently M.~Arkowitz and R.~Brown in \cite{ab04} characterized the Lefschetz number  of selfmaps. In the same spirit the first  author  and J. Weber  characterized the equivariant  Lefschetz number and the   Reidemeister trace 
of selfmaps in \cite{gw07}. 

Another concept which also plays a very important r\^ole is the concept of the ``local fixed point index'',    where
now we have enlarged the  domain where our function is defined.  A characterization of this function in terms of axioms for finite polyhedra 
was given by B. O'Neill in \cite{on53}. O'Neill's axioms later appeared in the book by R. Brown, \cite{br71}. For the differentiable case the problem was analyzed in
 \cite{fps04}. For the  related function,  the local Reidemiester trace for fixed points, the problem was considered by the second author in \cite{stae08a}. 

Only very recently,
in \cite{stae07a}, for differentiable orientable manifolds, the second author showed that the coincidence index is the unique integer-valued function which satisfies 3 axioms: additivity, homotopy invariance, and a normalization axiom stating that the total index taken over the whole domain space equals the coincidence Lefschetz number. This result used a direct generalization of the techniques in \cite{fps04} by Furi, Pera, and Spadini.
This work was then extended (still in the setting of differentiable orientable manifolds) to a uniqueness result for the Reidemeister trace for coincidences subject to 5 axioms in \cite{stae08a}.

The goal of the present paper is to prove the uniqueness of the coincidence index and the Reidemeister trace subject to axioms similar to those in the above work (specifically \cite{fps04}, \cite{stae07a}, and \cite{stae08a}), without using the differentiability or orientability assumptions. 
As we will see, dropping the differentiability assumption does little to change the character of the work, while orientation becomes the major point of focus. 

When both manifolds $M$ and $N$ are orientable we can choose an orientation for each manifold, and the index at a coincidence point $x$ of maps $f$ and $g$ is defined in terms of the way $f$ and $g$ carry the orientation from $x$ into the orientation at $f(x)=g(x)$. 
When our spaces are not orientable, the lack of a consistent local orientation for all points of the coincidence set will render this approach problematic.
In fact, it has been known for some time that there cannot be an integer valued function which behaves like the coincidence index in the setting of nonorientable manifolds.

The situation is fairly well-behaved when one of the maps (say $g$) is \emph{orientation true}: this means that for a loop $\gamma \subset M$, 
the loop $\gamma$ preserves a local orientation if and only if $g(\gamma)$ does. In this case, we show that there is in fact a way to consistently orient the coincidence set and to define an integer valued coincidence index. This index is proven to be unique subject to axioms similar to those used in \cite{fps04} and \cite{stae07a}.

Next we move to the general case, where neither map is assumed to be orientation true. In this case the difficulties in consistently orienting the spaces are not avoidable, and in some cases the index will not be integer-valued. In particular we divide the coincidence set into two types, nondegenerate and degenerate, depending on whether or not a sort of orientation true property holds locally. We will see that the nondegenerate coincidence points will have a coincidence index with value in $\Z$, while the degenerate coincidence points will have an index in $\Z_2 = \Z/2\Z$.

Thus in the case where neither map is assumed to be orientation true, our coincidence index will have its value in the group $\Z \oplus \Z_2$. This index is proven to be unique subject to axioms similar to those used in the orientation true case.

The choice of the set of values for the index (either $\Z$ or $\Z\oplus \Z_2$) is not arbitrary. In fact we show that any index with values in an abelian group $G$ which satisfies our axioms and and additional condition will in fact have values in a subgroup of $G$ isomorphic to $\Z$ or $\Z\oplus \Z_2$. 

One fundamental change in the character of the general case is that the index is somewhat less local. In particular the domain of the maps $f$ and $g$ becomes very important. It is possible, for example, to have maps $f,g:M \to N$ with open sets $U \subset V \subset M$ such that the index of $f$ and $g$ on $U$ is different from the index of $f|_V$ and $g|_V$ on $U$, where $f|_V,g|_V:V \to N$ are the restrictions of $f$ and $g$ to $V$. This can occur, for example, when $g$ is not orientation true and $M$ is nonorientable (so the index of $f$ and $g$ is in $\Z\oplus \Z_2$) but $V$ is orientable as a submanifold, with $g(V)$ contained in an orientable submanifold of $N$ (so $g|_V$ is orientation true and the index of $f|_V$ and $g|_V$ is in $\Z$).
Because of this, we must keep a careful account of the domain of the maps, and the value of the index will depend on this domain. Our focus, therefore, will be on \emph{local maps}, ones defined only on specific subsets of $M$.

We also extend the result of \cite{stae08a}, which is a uniqueness theorem for the Reidemeister trace. We show that, assuming $g$ is orientation true, the Reidemeister trace of maps $f$ and $g$ has value in $\Z\Reid(f,g)$, (where $\Reid(f,g)$ is the set of Reidemeister classes), and is unique subject to axioms similar to those in \cite{stae08a}. In the general case (when neither map is orientation true), we prove the uniqueness of a Reidemeister trace with value in $(\Z \oplus \Z_2)\Reid(f,g)$. We further show that $\Reid(f,g)$ splits as a disjoint union $\Reid(f,g) = \Reid_n(f,g) \sqcup \Reid_d(f,g)$ so that the value of the Reidemeister trace is always in $\Z\Reid_n(f,g) \oplus \Z_2 \Reid_d(f,g)$.

We begin in Section \ref{prelimsection} with a careful discussion of local orientations and properties of orientation true maps. In Section \ref{otruesection} we give our axioms and uniqueness result in the case where $g$ is orientation true. Section \ref{nonotruesection} drops this assumption and gives axioms and a uniqueness result in the general case. In Section \ref{valuessection} we show that the group of values of the index must be (isomorphic to) $\Z$ or $\Z\oplus \Z_2$. The results concerning the Reidemeister trace follow in Section \ref{rtsection}.
The bulk of the paper concerns the uniqueness of the index in various settings with respect to various axiom schemes. We conclude with a brief appendix concerning the existence, much of which is already documented in the literature.

\section{Local orientations and orientation true maps}\label{prelimsection}
We begin by discussing a suitable setting for the study of coincidences between  manifolds independent of whether the manifolds are  orientable or not.
This setting is guided  by the properties and knowledge  we have about coincidence theory, in particular when the manifolds involved are not both 
orientable.
 For  maps between two orientable manifolds, the coincidence theory has been well understood for some years (see \cite{sch55}) and as part
 of the data we fix one orientation 
for the domain and one orientation for the  target. 
We will extend this approach without assuming orientability in such a way that the  procedure reduces to the classical case if it happens that the  two manifolds are orientable. 
An intermediate situation is when the  two manifolds are not necessarily orientable but we restrict to
  the class of pairs of maps   where one of the maps is orientation-true. It is known that  the coincidence theory   for such pairs 
of maps has many similar properties to the  case where the manifolds are orientable. 

\subsection{  Local orientations} \label{orientsubsection}

Let $M$ be a path connected (or  equivalently connected in our particular case)
$n$-dimensional manifold. There is a unique  local system of coefficients
with local group $\Z$, which we call  the ``orientable 
bundle over $M$'', which exists whether or not the manifold is orientable.
This  local system is
provided once at a chosen point $x_0\in M$ we have the group  $\Z$ and a
representation $\pi_1(M, x_0) \to Aut(\Z)$ (see \cite{wh78}  Ch. VI, section 1 theorems (1.11) and (1.12)). The representation is simply given 
by sending an element $\alpha \in \pi_1(M,x_0)$
to the automorphism of $\Z$ which is multiplication by $\sign(\alpha)$. This sign is defined to be $\pm 1$ according to whether a local orientation is preserved or reversed when translated around the loop $\alpha$. Observe  that given an arbitrary point the transport
of a local  orientation along a path  provides a concrete construction of the unique (up to isomorphism)  bundle over  $M$ of the local system.

\begin{defn} 
The {\it orientable  bundle over $M$} is  the unique (up  to isomorphism) bundle
with   local group $\Z$ determined by the 
representation $\theta: \pi_1(M, x_0) \to \Aut(\Z)$ defined above 
where $x_0\in M$.
\end{defn}

If the manifold is orientable then the 
representation described above is the trivial homomorphism, and if the manifold is not orientable the representation is non-trivial. In 
any case we have  just one bundle.
Now we define a {\it local orientation at $x_0 \in M$}  as a chosen  generator of the local group
$H_n(U, U-x_0; \Z)$ for some small neighborhood $U$ of $x_0$. This generator we identify with 
$1$ of the local group $\Z$ at $x_0$. Note that if $M$ is path connected and orientable, then a choice of local orientation at a 
point is equivalent to a choice of a global orientation for $M$.

For the purpose of coincidence theory, based on our knowledge in the case where the manifolds are orientable,  we need to  
choose some kind of orientation  associated to the pair  $(M,N)$.
For  $(x_0,  y_0)\in M \times N$, we have two local orientations  at  $x_0$ and two local orientations 
at $y_0$. If we denote one local orientation  at a point $x$ by $O_{x}$ then the other one we  denote by
$-O_{x}$.  We say that two pairs 
$(O_{x_0}, O_{y_0})$,  $(O'_{x_0}, O'_{y_0})$ are equivalent if 
$(O_{x_0}, O_{y_0})=(\sigma O'_{x_0}, \sigma O'_{y_0})$ for some $\sigma \in\{ +1, -1\}$. In this case we write $[O_{x_0}, O_{y_0}]=[O'_{x_0}, O'_{y_0}]$.
So we have two equivalence classes at each pair of points $(x_0,y_0)$.  Then we define:

\begin{defn} A {\it local orientation for $(M, N)$ at $(x_0,  y_0)$} is a choice of one of these two classes defined above. 
 \end{defn}

\subsection{Orientation-true maps and coherent orientations} \label{otsubsection}
A map $f:M\to N$ is \emph{orientation true} when for any loop $\gamma$ in $M$, we have $\sign (\gamma) = \sign(f(\gamma))$. 

Here we state some properties of orientation true maps which are used and we apply some concepts from the previous subsection
to study coincidences of pairs where the second map is orientation true. 

First we show that the orientation true property is homotopy invariant.
\begin{lem} \label{othtp}
If $f,f':M \to N$ are homotopic and $f$ is orientation true, then $f'$ is orientation true.
\end{lem}

\begin{proof} The induced homomorphisms $f_{\#}:\pi_1(M,x_0) \to \pi_1(N,f(x_0))$ and $f'_{\#}: \pi_1(M,x_0) \to \pi_1(N,f'(x_0))$ satisfy
$f'_\#= \theta f_{\#}\theta^{-1}$, where $\theta$ is a path from $f(x_0)$ to $f'(x_0)$. This implies that  for every $\alpha \in \pi_1(M,x_0)$ we have 
\[ \sign(f'_{\#}(\alpha))= \sign(\theta * f_{\#}(\alpha) * \theta^{-1})= \sign(\theta * \theta^{-1}) \sign( f_{\#}(\alpha)) = \sign(f_{\#}(\alpha))\] and the result follows. 
 \end{proof}

Now let us consider a pair of maps $(f,g)$ from $M$ to $N$. For any open set $U\subset M$, let $\Coin(f,g,U) = \{x\in U \mid f(x) = g(x)\}$, and $\Coin(f,g) = \Coin(f,g,M)$.

For a path $\lambda$ from $a$ to $b$ and a local orientation $O$ at $a$, let 
$\lambda(O)$ be the orientation at $b$ given by transport of $O$ along $\lambda$.
Let $x_0,x_1 \in \Coin(f,g)$, and let $y_i=f(x_i)=g(x_i)$. Let $O_{x_i}$ be a local orientation at $x_i$, and similarly $O_{y_i}$ be a local orientation at $y_i$. We say that $[O_{x_0}, O_{y_0}]$ and $[O_{x_1}, O_{y_1}]$ are \emph{$g$-coherent} when there is a path $\lambda$ from $x_0$ to $x_1$ with $[O_{x_1},O_{y_1}] = [\lambda(O_{x_0}), g(\lambda)(O_{y_0})]$. The following shows that, when $g$ is orientation true, this formula will hold for \emph{any} path from $x_0$ to $x_1$.

\begin{lem}\label{anypath}
When $[O_{x_0}, O_{y_0}]$ and $[O_{x_1}, O_{y_1}]$ are $g$-coherent and $g$ is orientation true, we have $[O_{x_1},O_{y_1}] = [\gamma(O_{x_0}), g(\gamma)(O_{y_0})]$ for \emph{any} path $\gamma$ from $x_0$ to $x_1$.
\end{lem}
\begin{proof}
Since $[O_{x_0}, O_{y_0}]$ and $[O_{x_1}, O_{y_1}]$ are $g$-coherent, there is a path $\lambda$ from $x_0$ to $x_1$ with $[O_{x_1},O_{y_1}] = [\lambda(O_{x_0}), g(\lambda)(O_{y_0})]$. Now let $\gamma$ be another path from $x_0$ to $x_1$. We will show 
that $[\lambda(O_{x_0}),g(\lambda)(O_{y_1})] = [\gamma(O_{x_0}), g(\gamma)(O_{y_0})]$.

Let $\sigma = \pm 1$ be the sign such that $\lambda(O_{x_0}) = \sigma \gamma(O_{x_0})$, so that $\sigma = \sign(\gamma * \lambda^{-1})$. Since $g$ is orientation true this means $\sigma  = \sign(g(\gamma) * g(\lambda^{-1})) = 1$, which implies that $g(\lambda)(O_{y_0}) = \sigma g(\gamma)(O_{y_0})$. Thus we have $(\lambda(O_{x_0}),g(\lambda)(O_{y_0})) = (\sigma \gamma(O_{x_0)}, \sigma g(\gamma)(O_{y_0}))$ which means that $[\lambda(O_{x_0}),g(\lambda)(O_{y_0})] = [\gamma(O_{x_0}), g(\gamma)(O_{y_0})]$.
\end{proof}

Let $\C(f,g) = \{ (x,y) \mid f(x)=g(x)=y \}$.

\begin{defn}
A map  $\OO(x,y) = [O_x,O_y]$ which gives a class of pairs of local orientations at points $(x,y)\in \C(f,g)$ for which  $\OO(x_0,y_0)$ is $g$-coherent with  $\OO(x_1,y_1)$ for all $(x_i,y_i)\in \C(f,g)$ is called a \emph{$g$-coherent orientation of $\C(f,g)$}.
\end{defn}

Given a point $(x_0,y_0) \in \C(f,g)$ with $g$ orientation true, we can specify a $g$-coherent orientation $\OO$ of $\C(f,g)$ by choosing a specific 
orientation for $\OO(x_0,y_0) = [O_{x_0},O_{y_0}]$ (there are two possible choices), and then defining $\OO(x_i,y_i)$ to be $[\lambda(O_{x_0}),g(\lambda)(O_{y_0})]$, where $\lambda$ is a path from $x_0$ to $x_i$. This $\OO$ will be well defined (will not depend on choice of $\lambda$) by the above lemma. Since there are two possible values for $\OO(x_0,y_0)$, we have:
\begin{prop}\label{twoorientations}
When $g$ is orientation true, there are exactly two possible $g$-coherent orientations of $\C(f,g)$. These are each uniquely determined by their value at a single point.
\end{prop}

When we have homotopies $f \htp f'$ and $g \htp g'$ and a $g$-coherent orientation  $\OO$ of $\C(f,g)$, there is a naturally related choice of a $g'$-coherent orientation $\OO'$ of $\C(f',g')$. Let $(x,y)\in \C(f,g)$ and $(x',y')\in\C(f',g')$, and let $\OO(x,y) = [O_x, O_y]$. If $\gamma$ is a path from $x$ to $x'$ and $G$ is the homotopy of $g$ to $g'$, then let $\gamma_G([O_x,O_y]) = [\gamma(O_x), G(\gamma(t),t)(O_y)]$. 

\begin{defn}\label{Hreldefn}
Let $\OO$ and  $\OO'$ be $g$- (and $g'$- respectively) coherent orientations of $\C(f,g)$ and $\C(f',g')$ with $f\htp f'$ and $g\htp g'$, and let $G$ be the homotopy of $g$ to $g'$. We say that  $\OO'$ is \emph{$G$-related} to  $\OO$ when there are $(x,y)\in \C(f,g)$ and 
$(x',y') \in \C(f',g')$ with a path $\gamma$ from $x$ to $x'$ and  $\OO'(x',y') = \gamma_G(\OO(x,y))$.
\end{defn}

Such an orientation $\OO'$ is in fact unique:

\begin{lem}\label{hrel}
Let  $\OO$ be a $g$-coherent orientation of $\C(f,g)$ with $g$ orientation true, with $f\htp f'$ and $g\htp g'$, and let $G$ be the homotopy of $g$ to $g'$. Then there is a unique orientation $\OO'(x,y)$ of $\C(f',g')$ which is $G$-related to  $\OO$.
\end{lem}
In the proof we make repeated use of the following fact, which is an exercise:
\begin{prop}\label{pathprop}
If $G$ is a homotopy from $g$ to $g'$ and $\alpha$ is a path from $x$ to $x'$, then $G(\alpha(t),t)$ is a path from $g(x)$ to $g'(x')$ and
\[ G(\alpha(t),t) \htp g(\alpha) * G(x',t) \htp G(x,t) * g'(\alpha). \]
\end{prop}

\begin{proof}[Proof of Lemma \ref{hrel}]
Choose some $(x,y)\in \C(f,g)$ and $(x',y')\in \C(f',g')$. 
Making the definition $\OO'(x',y')=\gamma_G(\OO(x,y))$ suffices to define $\OO'$ on all of $\C(f',g')$ by Proposition \ref{twoorientations}. For the uniqueness, it suffices to show that this construction of $\OO'$ does not depend on the choice of $\gamma$ or on the choice of the points $x,x',y,y'$.

First we show that $\OO'$ is independent of the choice of path $\gamma$. Let $\bar \gamma$ be an alternative choice of path from $x$ to $x'$. Let $\OO(x,y) = [O_x, O_y]$. Our goal is to show that 
\[ [\gamma(O_x), G(\gamma(t), t)(O_y)] = [\bar \gamma(O_x), G(\bar \gamma(t), t)(O_y)]. \]

Let $\sigma = \pm 1$ be the sign such that $G(\gamma(t),t)(O_y) = \sigma G(\bar \gamma(t),t)(O_y)$. Then since transports along homotopic paths are equal, Proposition \ref{pathprop} gives $G(\gamma(t),t) (O_y) = g(\gamma)*G(x',t)(O_y)$ and similarly with $\bar \gamma$. Thus we have
\[ g(\gamma) * G(x',t)(O_y) = \sigma g(\bar \gamma) * G(x',t)(O_y), \]
and so $\sigma = \sign(g(\gamma^{-1})*g(\bar \gamma))$. Since $g$ is orientation true this means $\sigma = \sign(\gamma^{-1}*\bar \gamma)$, which means $\gamma(O_x) = \sigma \bar \gamma(O_x)$. Thus we have
\[ (\gamma(O_x), G(\gamma(t), t)(O_y)) = (\sigma \bar \gamma(O_x), \sigma G(\bar \gamma(t), t)(O_y)), \]
which completes the argument.

Now we show that $\OO'$ does not depend on the choice of points $x,y,x',y'$. Our orientation $\OO'$ was constructed by starting with $\OO(x,y)$ and carrying this orientation through $G$ on the path $\gamma$ to an orientation at $(x',y')$. 

Let $(\bar x, \bar y) \in \C(f,g)$ and $(\bar x', \bar y') \in \C(f',g')$, and we will construct an orientation $\OO'$ of $\C(f',g')$ by starting with $\OO(\bar x, \bar y)$ and carrying this orientation through $G$ to $(\bar x', \bar y')$. Then we must show that $\OO' = \bar \OO'$. This we do by showing that they agree at the point $(\bar x',\bar y')$. 

In our construction above, we have $\OO'(x',y') = [\gamma(O_x), G(\gamma(t),t)(O_y)]$. Let $\lambda'$ be a path from $x'$ to $\bar x'$, and then by coherence of $\OO'$ we have
\begin{align}
\OO'(\bar x', \bar y') &= [\lambda' (\gamma(O_x)), g'(\lambda') (G(\gamma(t),t)(O_y))] \notag \\
&= [(\gamma * \lambda')(O_x), (G(\gamma(t),t) * g'(\lambda'))(O_y))].\label{nobarO}
\end{align}

For $\bar \OO'(\bar x',\bar y')$, we start with $\OO(\bar x, \bar y) = [\lambda(O_x), g(\lambda)(O_y)]$, where $\lambda$ is a path from $x$ to $\bar x$. Let $\bar \gamma = \lambda^{-1} * \gamma * \lambda'$, this is a path from $\bar x$ to $\bar x'$. Then carrying $[\lambda(O_x), g(\lambda)(O_y)]$ through $G$ along $\bar \gamma$ gives
\begin{align}
\bar \OO'(\bar x', \bar y') &= [\bar \gamma (\lambda(O_x)), G(\bar \gamma(t),t)( g(\lambda) (O_y))] \notag \\
&= [(\lambda * \bar \gamma) (O_x), (g(\lambda) * G(\bar \gamma(t),t)) (O_y))] \notag \\
&= [(\gamma * \lambda') (O_x), (g(\lambda) * G(\bar \gamma(t),t)) (O_y))]. \label{barO}
\end{align}

To show equality of \eqref{nobarO} and \eqref{barO} it suffices to show that the paths $G(\gamma(t),t) * g'(\lambda')$ and $g(\lambda)*G(\bar \gamma(t),t)$ are homotopic. Using Proposition \ref{pathprop} repeatedly gives
\begin{align*}
g(\lambda)*G(\bar \gamma(t),t) &\htp g(\lambda) * g(\bar \gamma) * G(\bar x',t) \htp g(\lambda * \bar\gamma) * G(\bar x',t) \\
&\htp g(\gamma * \lambda') * G(\bar x',t) \htp G(x,t) * g'(\gamma * \lambda') \\ 
&\htp G(x,t) * g'(\gamma) * g'(\lambda') \htp G(\gamma(t),t) * g'(\lambda')
\end{align*}
as desired.

\end{proof}

When $f,g:M \to M$ are selfmaps of orientable manifolds, we can relate the local orientations at points in the domain to the local orientations at points in the codomain.

\begin{defn} 
Let $f,g:M \to M$ and let $\OO$ be a $g$-coherent orientation of $\C(f,g)$.
We say that $\OO$ is \emph{oriented consistently} if, given any point $(x,y) \in \C(f,g)$ with $\OO(x,y) = [O_x,O_y]$, we 
have $\lambda(O_x) = O_y$ for any (and thus every) path $\lambda$ from $x$ to $y$.
\end{defn} 

Intuitively this indicates that the orientation chosen in the domain space is the same as that chosen in the codomain space.
If $\C(f,g)$ consists of a single point $(x,x)$, then $\OO$ is oriented consistently if and only if $\OO(x,x) = [O_x,O_x]$ for some local orientation $O_x$ at $x$.

\begin{defn}
For a homeomorphism $h:U \subset M \to N$ with a point $x\in U$ and local orientations $O_x, O_{h(x)}$, define $\sign (h,[O_x,O_{h(x)}])$ to be $+1$
if $h_*(O_{x}) = O_{h(x)}$, and $-1$ otherwise. 
\end{defn}

Note that this sign is well defined with respect to the two representations of the equivalence class $[O_{x}, O_{h(x)}]$. Also note that when $h$ is a diffeomorphism $h:\R^n \to \R^n$ and $O_x$ and $O_{h(x)}$ are the standard orientations of $\R^n$, then $\sign(h,[O_x,O_{h(x)}]) = \sign(\det(dh_x))$.

\section{The index for  pairs of maps where one of the maps is orientation-true}\label{otruesection}

To give our coincidence index a truly local setting, we will consider only \emph{local maps}. A local map $f$ of $M \to N$ is a continuous map whose domain $\dom f$ is an open subset of $M$, and whose image is in $N$. Note that all results of the previous section concerning orientations apply to pairs of local maps which have the same connected domain.

\subsection{Pairs of maps where the second is orientation true}
The coincidence index in this subsection will be defined on the following set:

\begin{defn}\label{tuples} 
Let  $\adm$  be the set of all tuples  $(f,
g, U, \OO)$, where $f,g$ are local maps of $M \to N$ for some manifolds $M$ and $N$ of the same dimension, the domains of $f$ and $g$ are the same connected set, the set $U\subset M$ is an open subset of the domain of $f$ and $g$, the set $\Coin(f, g, U)$ is compact, $g$ is 
orientation true, and $\OO$ is a $g$-coherent orientation of $\C(f,g)$. 
\end{defn}

Note that for $(f,g,U,\OO) \in \adm$ with $U \subset M$, we do not necessarily assume that $f$ and $g$ are restrictions of globally defined functions $\hat f,\hat g:M \to N$. 
In fact, in this section, the local index will not depend on whether or not there exist such extensions $\hat f, \hat g$. Further, as long as $\hat g$ remains orientation true, the local index will be the same for the tuples $(f,g,U,\OO)$ and $(\hat f, \hat g,U,\hat \OO)$ where $\hat \OO$ is the appropriate extension of $\OO$.

The set $\adm$
is called the set of \emph{admissible tuples}.
We say that a pair of homotopies $F, G: (\dom f) \times [0,1] \to N$ is
\emph{admissible} when $\Coin(F,G,U\times [0,1])$ is compact in 
$U \times [0,1]$. We say that two tuples 
$(f, g, U, \OO),  (f', g', U,\OO') \in
\adm$ are \emph{admissably homotopic} if there is a pair of
admissible homotopies taking $f$ to $f'$ and $g$ to $g'$ and 
$\OO'$ is $G$-related to $\OO$, where $G$ is the homotopy between 
$g$ and $g'$.

Our main result is that there is at most one function $\iota:  \adm \to \R$ satisfying
the following
axioms (compare to the axioms of \cite{fps04} and \cite{stae07a}):
\begin{axiom}[Additivity axiom]
Given $(f,g,U, \mathbb O)\in \adm$, if \ $U_1$ and
$U_2$ are disjoint open
subsets of $U$ with $\Coin(f,g, U) \subset U_1 \cup U_2$, 
then
\[ \iota(f,g,U, \mathbb O) = \iota(f,g,U_1, \mathbb O) +
\iota(f,g,U_2, \mathbb O). \]
\end{axiom}

\begin{axiom}[Homotopy axiom]
If $(f, g, U, \mathbb O) \in \adm$ and $(f',
g', U, \mathbb O')\in \adm$ are
admissibly homotopic, then
\[ \iota(f, g, U, \mathbb O) = \iota(f', g',U, \mathbb O'). \]
\end{axiom}

\begin{axiom}[Normalization axiom]
Let $(c,g,U,\OO) \in \adm$ with $c|_U$ a constant map with constant value $c$ and $g|_U$ an embedding with $g(x) = c$. If $\sign(g,\OO(x,c)) = 1$, then 
\[ \iota(c,g,U,\OO) = 1. \]
\end{axiom}

One useful property follows immediately from the additivity axiom:
\begin{prop}[Excision property]
Let $\iota:\adm \to \R$ satisfy the Additivity axiom.
Let $(f,g,U,\OO) \in \adm$ and let $U' \subset U$ be an open set with $\Coin(f,g,U) \subset U'$. Then
\[ \iota(f,g,U,\OO) = \iota(f,g,U',\OO). \]
\end{prop}
\begin{proof}
Applying the Additivity axiom to the disjoint union $\Coin(f,g,\emptyset) \subset \emptyset \cup \emptyset$ gives
\[ \iota(f,g,\emptyset,\OO) + \iota(f,g,\emptyset,\OO) = \iota(f,g,\emptyset,\OO), \]
and so $\iota(f,g,\emptyset,\OO) = 0$.

The excision property follows from the Additivity axiom applied to the union $\Coin(f,g,U) \subset U' \cup \emptyset$, which gives
\[ \iota(f,g,U,\OO) = \iota(f,g,U',\OO) + \iota(f,g,\emptyset, \OO)
=  \iota(f,g,U',\OO). \qedhere \]
\end{proof}

Let $\adm_{\R^n}$ be the set of admissable tuples $(f,g,U,\OO)$ of maps whose domain and codomain are subsets of $\R^n$, and $\OO$ is consistently oriented. From Lemmas 11 and 12 of \cite{stae08a}, the coincidence index is unique on $\adm_{\R^n}$ subject to additivity and homotopy axioms formulated similarly to ours, and a ``weak normalization'' axiom stating that the index of a constant map with the identity map is 1. (The requirement that the orientation be consistently oriented is implicit in \cite{stae08a}: a single orientation on $\R^n$ is fixed throughout the paper, making it impossible in the setting of maps $\R^n \to \R^n$ to choose different orientations in the domain and codomain.)

Let $\ind:\adm_{\R^n} \to \R$ denote this unique coincidence index for $\R^n$. If $\iota$ satisfies our three axioms, then the restriction of $\iota$ to $\adm_{\R^n}$ satisfies the three axioms of \cite{stae08a} (our Normalization axiom implies the weak normalization axiom when restricted to $\adm_{\R^n}$) and we have $\iota(f,g,U,\OO) = \ind(f,g,U,\OO)$ when $(f,g,U,\OO) \in \adm_{\R^n}$.

Next we show that if $\iota$ satisfies our three axioms, then its value on a euclidean neighborhood of a single coincidence point agrees with $\ind$ when we move the setting into $\R^n$ by applying charts.

\begin{lem}\label{naturality}
Let $(f,g,U,\OO) \in \adm$ with $\Coin(f,g,U)$ consisting of a single coincidence point $x$ such that $U$ is a euclidean neighborhood of $x$ and $f(U) \cup g(U)$ is contained in a euclidean neighborhood $W$ of $y=f(x)=g(x)$. Let $j:U \to \R^n$ and $h:W \to \R^n$ be homeomorphisms, and let $\bar f = h \circ f \circ j^{-1}$ and $\bar g = h \circ g \circ j^{-1}$. 

Let $\OO(x,y) = [O_x,O_y]$, and 
define an orientation $\bar \OO$ of $\C(\bar f, \bar g)$ by 
$\bar \OO(j(x),h(y)) = [j_*(O_x), h_*(O_y)]$.
If $\bar \OO$ is oriented consistently, and $\iota$ satisfies the three axioms, then 
\[ \iota(f,g,U,\OO) = \ind(\bar f, \bar g, \R^n, \bar \OO). \]
\end{lem}
\begin{proof}
Throughout, we will use the bar to indicate application of $j$ and $h$ appropriately as in the definitions of $\bar f, \bar g$, and $\bar \OO$.
Let $\mathcal U \subset \adm$ be the set of all tuples $(s,t,Z,\mathbb A)$ with $Z \subset U$ and $\bar {\mathbb A}$ oriented consistently. Then there is a bijection $\omega:\mathcal U \to \adm_{\R^n}$ given by $\omega(s,t,Z,\mathbb A) = (\bar s, \bar t, j(Z), \bar {\mathbb A})$, and so
we have a map $\iota \circ \omega^{-1}: \adm_{\R^n} \to \R$. 

We now note that $\iota \circ \omega^{-1}$ satisfies the homotopy, additivity, and weak normalization axioms of \cite{stae08a}, and thus is equal to $\ind$, the unique coincidence index on $\adm_{\R^n}$. The homotopy and additivity axioms are clear, but the weak normalization axiom deserves some comment. 

We must show that $\iota \circ \omega^{-1}(c,\id,\R^n, \mathbb B) = 1$ where $c$ is a constant map with constant value $c\in \R^n$, and $\mathbb B$ is oriented consistently. From the definition of $\omega$ we have 
\[ \omega^{-1}(c,\id,\R^n,\mathbb B) = (h^{-1}\circ c \circ j, h^{-1}\circ j, U, \mathbb A),\]
where $\bar {\mathbb A} = \mathbb B$. The tuple on the right above has $j^{-1}(c)$ as its only coincidence point, with common value $h^{-1}(c)$. The first map is a constant and the second is an embedding, so the Normalization axiom will apply to show that $\iota(h^{-1}\circ c \circ j, h^{-1}\circ j, U, \mathbb A) = 1$, provided that
$\sign(h^{-1}\circ j,\mathbb A(j^{-1}(c),h^{-1}(c)) = 1$.

Let $\mathbb A(j^{-1}(c), h^{-1}(c)) = [O_{j^{-1}(c)}, O_{h^{-1}(c)}]$, and we will show that $(h^{-1}\circ j)_* O_{j^{-1}(c)} = O_{h^{-1}(c)}$, and thus the above sign is 1. Since $\bar {\mathbb A} = \mathbb B$, we have $\mathbb B(c,c) = [j_*(O_{j^{-1}(c)}), h_*(O_{h^{-1}(c)})]$, and since $\mathbb B$ is oriented consistently this means $j_*(O_{j^{-1}(c)})= h_*(O_{h^{-1}(c)})$. Now we compute
\[ (h^{-1}\circ j)_* (O_{j^{-1}(c)}) = h^{-1}_* (j_*(O_{j^{-1}(c)})) = h^{-1}_* (h_*(O_{h^{-1}(c)})) = O_{h^{-1}(c)} \]
as desired.

We have shown that $\iota\circ \omega^{-1}:\adm_{\R^n} \to \R$ satisfies the three axioms of \cite{stae08a}, and thus it equals $\ind$, the unique coincidence index on $\R^n$. Thus we have $\iota(f,g,U,\OO) = \ind \circ \omega(f,g,U,\OO) = \ind(\bar f, \bar g, j(U), \bar \OO)$ as desired.
\end{proof}

For our uniqueness theorem, we will make use of Lemma 15 from \cite{stae07a}, showing that our tuples
may be changed by admissible homotopy to have isolated coincidences:

\begin{lem}\label{approx}
Let $(f,g,U,\OO) \in \adm$, and let $V \subset U$ be an
open subset containing $\Coin(f,g,U)$ with compact closure $\bar V
\subset U$. Then $(f,g,V,\OO)$ is admissably homotopic to some
$(f',g',V,\OO')$, where $f'$ and $g'$ have isolated coincidence
points in $V$.
\end{lem}

Our strategy for the main result is to change the maps by homotopy, first so that they have isolated coincidence points, and then use charts and Lemma \ref{naturality} to determine the index.

\begin{thm}\label{otrueuniqueness}
There is at most one function $\iota: \adm \to \R$ which
satisfies the axioms.
\end{thm}
\begin{proof}
Let $(f,g,U, \mathbb  O) \in \adm$, and let $\iota:\adm \to \R$ satisfy the axioms. By the Homotopy axiom and Lemma \ref{approx}, we may 
assume without loss of generality that $\Coin(f,g,U)$ is a finite set
of isolated coincidence points. For each such point $x$, let $U_x$ be a neighborhood of $x$ containing no other coincidences. Then by the additivity axiom we have
\[ \iota(f,g,U,\OO) = \sum_{x\in \Coin(f,g,U)} \iota(f,g,U_x,\OO). \]

Take a particular coincidence point $x \in \Coin(f,g,U)$, and we will show that the axioms alone determine the value of $\iota(f,g,U_x,\OO)$. Let $y=f(x)=g(x)$, and let $\OO(x,y) = [O_x,O_y]$.
Since $\Coin(f,g,U_x)$ is a single coincidence point, we may shrink $U_x$ by the excision property and assume that there is a euclidean neighborhood $W_x$ containing $f(U_x) \cup g(U_x)$ and 
homeomorphisms $h:U_x \to \R^n$ and $j:W_x \to \R^n$ with $h(x)=j(y)=0$. Furthermore we may assume (perhaps by composing with an orientation reversing selfmap of $\R^n$) that $h_*(O_x) = j_*(O_y) = O_0$, where $O_0$ is the standard local orientation of $\R^n$ at the origin. Let $\bar f_x = j \circ f \circ h$ and $\bar g_x = j \circ g \circ h$, and define $\bar \OO$ as in Lemma \ref{naturality}, and we have
\[ \iota(f,g,U_x,\OO) = \ind(\bar f_x, \bar g_x, \R^n, \bar \OO). \]
Thus the value of $\iota(f,g,U_x,\OO)$ is unique, and summing over $x$ shows that $\iota(f,g,U,\OO)$ is unique.
\end{proof}

The proof above immediately leads to some corollaries.
First, note that $\iota$ is computed as a sum of certain values of the classical coincidence index $\ind$. Since this index is integer valued, we have:
\begin{cor}\label{zvalued}
Any function $\iota:\adm \to \R$ satisfying the axioms will have values in $\Z$.
\end{cor}

Note that if we interchange the roles of $f$ and $g$ (provided that both are orientation true and $\OO$ is both $f$- and $g$-coherent) the value of $\iota$ will be computed as
\[ \iota(g,f,U,\OO) = \sum_{x \in \Coin(f,g)} \ind(\bar g_x, \bar f_x, \R^n, \bar \OO). \]
A well-known formula for the coincidence index for orientable manifolds (in this case, it is just the Lefschetz number) gives $\ind(\bar g_x, \bar f_x,\R^n, \bar \OO) = (-1)^n \ind(\bar f_x, \bar g_x, \R^n, \bar \OO)$, and thus we have
\begin{cor}\label{swap}
Let $\iota:\adm \to \R$ satisfy the axioms, and let $(f,g,U,\OO) \in \adm$ such that $(g,f,U,\OO)$ is also in $\adm$. Then
\[ \iota(f,g,U,\OO) = (-1)^n \iota(g,f,U,\OO). \]
\end{cor}

Another corollary which we can obtain from the proof of Theorem \ref{otrueuniqueness} is that the value of $\iota$ does not depend on the domain of the maps. The value of $\iota(f,g,U,\OO)$ depends only on the local behavior of $f$ and $g$ near the coincidence points. Thus we have the following, referred to in \cite{fps04} as the ``Localization property".
\begin{cor}[Localization property]\label{localization}
Let $(f, g, U, \OO) \in \adm$ with $\dom f = \dom g = V$, and let $(f|_W, g|_W, U, \OO|_W) \in \adm$ be the tuple of the restrictions to an open connected set $W$ with $U \subset W \subset V$.
If $\iota:\adm \to \R$ satisfies the axioms, then 
\[ \iota(f,g,U,\OO) = \iota(f|_W, g|_W, U, \OO|_W). \]
\end{cor}
We point out the above specifically because the Localization property will no longer hold when we drop the orientation true assumption.

Recall by Proposition \ref{twoorientations} that there are exactly two orientations of $\C(f,g)$ when $g$ is orientation true. If $\OO$ is one such orientation, let $-\OO$ be the other one. By the proof above, the value of $\iota$ is computed using Lemma \ref{naturality}, and the value depends on the classical coincidence index on $\R^n$. It is not hard to check that changing $\OO$ to $-\OO$ has the effect of reversing the sign of this value. Thus we have
\begin{cor}\label{flipsign}
If $\iota$ satisfies the axioms, then
\[ \iota(f,g,U,\OO) = -\iota(f,g,U,-\OO). \]
That is, the absolute value of $\iota$ does not depend on the choice of orientation $\OO$.
\end{cor}

\subsection{The case where at least one map is orientation true}\label{forgsubsection}
In the construction of our set $\adm$ of tuples, we always require the second map to be orientation true. This is merely a convention, and in this section we briefly discuss the case where the first map, rather than the second, is assumed to be orientation true, and the general setting where at least one map (either one) is orientation true.

Let $\adm'$ be the set of tuples $(f,g,U,\OO)$ where $\Coin(f,g,U)$ is compact, $f$ is orientation true, and $\OO$ is a $f$-coherent orientation of $\C(f,g)$. That is, $(f,g,U,\OO) \in \adm'$ if and only if $(g,f,U,\OO) \in \adm$.
We say that $(f,g,U,\OO), (f',g',U,\OO') \in \adm'$ are admissably homotopic when there are admissible homotopies of $f$ to $f'$ and $g$ to $g'$ and $\OO'$ is $F$-related to $\OO$, where $F$ is the homotopy from $f$ to $f'$. 

Our Additivity and Homotopy axioms in this setting are exactly the same as for $\adm$, but the normalization must change. Our uniqueness result is:
\begin{thm}\label{uniquenessprime}
There is at most one function $\iota':\adm' \to \R$ satisfying three axioms: 
\begin{itemize}
\item (Additivity) 
Given $(f,g,U, \mathbb O)\in \adm'$, if \ $U_1$ and
$U_2$ are disjoint open
subsets of $U$ with $\Coin(f,g, U) \subset U_1 \cup U_2$, 
then
\[ \iota'(f,g,U, \mathbb O) = \iota'(f,g,U_1, \mathbb O) +
\iota'(f,g,U_2, \mathbb O). \]

\item (Homotopy)
If $(f, g, U, \mathbb O) \in \adm'$ and $(f',
g', U, \mathbb O')\in \adm'$ are
admissably homotopic, then
\[ \iota'(f, g, U, \mathbb O) = \iota'(f', g',U, \mathbb O'). \]

\item (Normalization)
Let $(f,c,U,\OO) \in \adm'$ with $c|_U$ a constant map with constant value $c$ and $f|_U$ an embedding with $f(x) = c$. If $\sign(f,\OO(x,c)) = 1$, then 
\[ \iota'(f,c,U,\OO) = (-1)^n. \]
\end{itemize}
\end{thm}
\begin{proof}
Let $\iota'$ satisfy the three axioms, and define $\iota:\adm \to \R$ by $\iota(f,g,U,\OO) = (-1)^n \iota'(g,f,U,\OO)$. This function $\iota$ satisfies the axioms of the previous section for maps $\adm \to \R$, and thus must be unique, and this specifies uniquely the function $\iota'$.
\end{proof}

We thus have two functions, $\iota$ and $\iota'$, purporting to be the unique coincidence index on the two sets $\adm$ and $\adm'$. 
By the proof above we have $\iota(f,g,U,\OO) = (-1)^n \iota'(g,f,U,\OO)$, and by Corollary \ref{swap}, if $(f,g,U,\OO) \in \adm \cap \adm'$ we have $\iota(f,g,U,\OO) = (-1)^n \iota(g,f,U,\OO)$. Thus we have
\begin{thm}\label{agreement}
Let $\iota:\adm \to \R$ and $\iota':\adm' \to \R$ satisfy their respective axiom schemes, and let $(f,g,U,\OO) \in \adm \cap \adm'$. Then 
\[ \iota(f,g,U,\OO) = \iota'(f,g,U,\OO). \]
\end{thm}

Since $\iota$ and $\iota'$ agree on the intersection of their domains, we actually have a unique index defined on the class $\hat \adm = \adm \cup \adm'$ of all tuples where at least one map is orientation true and the orientation is appropriately coherent. 
\begin{cor}
There is at most one function $\hat \iota:\hat \adm \to \R$ satisfying the axioms:
\begin{itemize}
\item (Additivity) 
Given $(f,g,U, \mathbb O)\in \hat \adm$, if \ $U_1$ and
$U_2$ are disjoint open
subsets of $U$ with $\Coin(f,g, U) \subset U_1 \cup U_2$, 
then
\[ \hat \iota(f,g,U, \mathbb O) = \hat \iota(f,g,U_1, \mathbb O) +
\hat \iota(f,g,U_2, \mathbb O). \]

\item (Homotopy)
If $(f, g, U, \mathbb O) \in \hat \adm$ and $(f',
g', U, \mathbb O')\in \hat \adm$ are
admissably homotopic (either as elements of $\adm$ or $\adm'$), then
\[ \hat \iota(f, g, U, \mathbb O) = \hat \iota(f', g',U, \mathbb O'). \]

\item (Normalization)
Let $c|_U$ be a constant map with constant value $c$ and $f|_U$ an embedding with $f(x) = c$ and $\sign(f,\OO(x,c)) = 1$. Then:
\begin{align*}
\hat \iota(c,f,U,\OO) &= 1 \quad \text{for } (c,f,U,\OO) \in \adm, \text{ and} \\
\hat \iota(f,c,U,\OO) &= (-1)^n \quad \text{for } (f,c,U,\OO) \in \adm'.
\end{align*}
\end{itemize}
\end{cor}

\section{The case of arbitrary  maps between arbitrary manifolds
without boundary }\label{nonotruesection}
Now we attempt to mimic the definitions and arguments of the previous section without the assumption that one map is orientation true. 
\subsection{Local orientations in the general case}
We begin as in the orientation-true case by discussing coincidence sets and their local orientations. 
\begin{defn}
Let $\mathcal D_0$ be the set of all tuples $(f,g,U,V)$, where $f,g$ are local maps of $M \to N$ with connected domain $V$, the spaces $M$ and $N$ are manifolds without boundary of the same dimension, the set $U$ is an open subset of $V$, and $\Coin(f,g,U)$ is compact.
\end{defn}

Though the local maps $f,g$ implicitly carry their domain, we include the set $V$ in our tuples for emphasis to distinguish the setting from that the previous section. The fundamental difference in the non-orientation true case is that the index will depend on the domain set $V$. When one map is orientation true, this is not the case by the Localization property.
Without the assumption that $g$ be orientation true, it is possible when $W\subset V$ for the tuples $(f|_W,g|_W,U,W)$, $(f,g,U,V)$ to have different values for the coincidence index.

To emphasise the importance of the domain set $V$ for a tuple $(f,g,U,V)$, we will write $\C(f,g)=\C(f,g,V)$. 
Next we define the concept of {\it degenerate and nondegenerate coincidence points of a  tuple}.
We say that a coincidence $(x,y)\in \C(f,g,V)$ is \emph{degenerate}
when there is a loop $\theta \in \pi_1(V,x)$ with
$f_\#(\theta)=g_\#(\theta)$ and $\sign(\theta) \neq
\sign(g_\#(\theta))$. Otherwise, we say that $(x,y)$ is
\emph{nondegenerate}. Sometimes we will refer to a coincidence point $x\in \Coin(f,g,U)$ as being degenerate or nondegenerate according to the degeneracy of $(x,y)$.
Note that when $g$ is orientation true, all coincidences are nondegenerate.

Let $\C_d(f,g,V)\subset \C(f,g,V)$ be the set of degenerate coincidences, and $\C_n(f,g,V)\subset \C(f,g,V)$ be the nondegenerate coincidences. This partitioning into degenerate and nondegenerate coincidences depends strongly on the domain:
when $W \subset V$, we have $\C_d(f|_W,g|_W,W) \subset \C_d(f,g,V)$ and $\C_n(f,g,V) \subset \C_n(f|_W,g|_W,W)$. 

When $g$ is not assumed to be orientation true, the coherence condition for orientations at points of $\Coin(f,g, U)$ is not as well behaved as in Subsection \ref{otsubsection}. Two orientations $[O_x, O_y]$ and
$[O_{x'},O_{y'}]$ may be $g$-coherent with respect to some paths from $x$
to $x'$, but not with respect to others. 
We will see, however, a weaker notion of coherence can be used among the nondegenerate coincidences. 

For $(f,g,U,V) \in \admm_0$,
given two coincidence points $x_0,x_1 \in \Coin(f,g, U)$, we say they are in the same \emph{coincidence class} if there is a path $\gamma$
in $V$ 
from $x_0$ to $x_1$ with $f(\gamma) \htp g(\gamma)$. Such a path we will call a \emph{Nielsen path}.

Let $(x_0,y_0),(x_1,y_1) \in \C_n(f,g,V)$, with $x_0, x_1$ in the same coincidence class,
and let $O_{x_i}, O_{y_i}$ be local orientations at $x_i$ and
$y_i$. Then we say that $[O_{x_0}, O_{y_0}]$ and $[O_{x_1},
O_{y_1}]$ are \emph{Nielsen coherent} when there is a Nielsen path $\gamma$ from
$x_0$ to $x_1$ with $[O_{x_1},O_{y_1}] = [\gamma(O_{x_0}),
f(\gamma)(O_{y_0})] = [\gamma(O_{x_0}),
g(\gamma)(O_{y_0})]$.

A \emph{Nielsen coherent orientation} of $\C(f,g,V)$ is a function $\OO(x,y) =
[O_x,O_y]$ for which $\OO(x_0,y_0)$ and $\OO(x_1,y_1)$ are Nielsen coherent when $x_0$ and $x_1$ 
are in the same coincidence class.

If $g$ is orientation true and $\OO$ is $g$-coherent, then all coincidences are nondegenerate and by Lemma \ref{anypath} $\OO$ will be Nielsen coherent. Thus Nielsen coherence of an orientation is weaker than $g$-coherence when $g$ is orientation true.
A weaker form of Lemma \ref{anypath} will hold for Nielsen coherence within coincidence classes.
\begin{lem}
When $x_0, x_1$ are nondegenerate coincidence points of a  tuple \break $(f,g,U,V) \in \admm_0$ in the same coincidence class and $[O_{x_0},O_{y_0}]$ and $[O_{x_1}, O_{y_1}]$ are Nielsen coherent, we have $[O_{x_1},O_{y_1}] = [\gamma(O_{x_0}), g(\gamma)(O_{y_0})]= [\gamma(O_{x_0}), f(\gamma)(O_{y_0})]$ for any Nielsen path $\gamma$ from $x_0$ to $x_1$.
\end{lem}
\begin{proof}
Let $\bar \gamma$ be another Nielsen path from $x_0$ to $x_1$, and let $\theta = \bar \gamma * \gamma^{-1}$. The proof of 
Lemma \ref{anypath} will apply provided that $\sign \theta = \sign g(\theta)$. Since $\gamma$ and $\bar \gamma$ are 
both Nielsen paths we will have $f(\theta) \htp g(\theta)$, and thus since $x_0$ and $x_1$ are nondegenerate 
we have $\sign \theta = \sign g(\theta) = \sign f(\theta)$ as desired.
\end{proof}

Thus a Nielsen coherent orientation is uniquely determined once we specify its value at one 
point from each coincidence class. 

Our present setting requires a modified form of homotopy-relatedness.

Let $\OO$ be a Nielsen coherent orientation of $\C(f,g,V)$ associated to  the tuple $(f,g,U,V)$, 
and $\OO'$ be a Nielsen coherent orientation of $\C(f',g',V)$ with $f'\htp f$ and $g' \htp g$. Let
$F$ and $G$ be the homotopies taking $f$ to $f'$ and $g$ to $g'$. Recall that 
if $(x,y)\in \C(f,g,V)$ and $(x',y') \in \C(f',g',V)$ 
and $\gamma$ is a path in $V$ from $x$ to $x'$, then 
\[ \gamma_G(\OO(x,y))= \gamma_G([O_x,O_y]) = [\gamma(O_x), G(\gamma(t),t)(O_y)] \]
is an orientation at $(x',y')$.

\begin{defn}\label{nielsenFGrel}
Let $(f,g,U,V)\in \admm_0$, and let $\OO$ be a Nielsen coherent orientation of $\C(f,g,V)$ and $\OO'$
be a Nielsen coherent orientation of $\C(f',g',V)$ with $f'\htp f$ and $g' \htp g$. Let
$F$ and $G$ be the homotopies taking $f$ to $f'$ and $g$ to $g'$. 

We
say that $\OO'$ is \emph{$(F,G)$-related} to $\OO$ if: whenever $(x,y) \in \C_n(f,g,V)$ and $(x',y')\in \C_n(f',g',V)$ with a path $\gamma$ 
from $x$ to $x'$ in $V$ such that $F(\gamma(t),t) = G(\gamma(t),t)$, we have $\OO'(x',y')=\gamma_F(\OO(x,y)) = \gamma_G(\OO(x,y))$.
\end{defn}

This notion of relatedness is weaker than $G$-relatedness as defined in Definition \ref{Hreldefn}. Specifically, when $g$ is orientation true and $\OO$ is $g$-coherent and $\OO'$ is $g'$-coherent, then if $\OO$ and $\OO'$ are $G$-related, they are automatically $F,G$-related.

\begin{lem}
As described above, $F,G$-relatedness is well defined. That is, it
does not depend on the choice of path $\gamma$.
\end{lem}
\begin{proof}
Let $\OO$ be a Nielsen coherent orientation of $\Coin(f,g,V)$ and let $\OO'$ be a
Nielsen coherent orientation of $\Coin(f',g',V)$ which is $F,G$-related to $\OO$. Say that we have points
$(x,y)\in \C_n(f,g,V)$ with $\OO(x,y)=[O_x,O_y]$ and
$(x',y')\in\C_n(f',g',V)$ with a path $\gamma(t)$ in $M$ from $x$ to $x'$ such
that $F(\gamma(t),t) = G(\gamma(t),t)$ for all $t$, and then by
$F,G$-relatedness we have $\OO'(x',y') = [\gamma(O_x),
G(\gamma(t),t)(O_y)]$.

Let $\bar \gamma$ be another path from $x$ to $x'$ with $F(\bar
\gamma(t),t) = G(\bar \gamma(t),t)$, and let $\bar \OO'$ be the same
as $\OO'$, except using $\bar \gamma$ instead of $\gamma$ to define
$\bar \OO'(x',y')$. We will show that $\bar \OO'(x',y') =
\OO'(x',y')$.

The proof of the corresponding statement from Lemma \ref{hrel} uses
the fact that $g$ is orientation true on the loop $\gamma * \bar
\gamma^{-1}$. The same proof will work here provided that we show
$\sign (\gamma * \bar \gamma^{-1}) = \sign(g(\gamma*\bar
\gamma^{-1}))$. 

Let $\lambda = \gamma * \bar \gamma^{-1}$, and using
Proposition \ref{pathprop} gives
\begin{align*}
f(\lambda) &= f(\gamma) * f(\bar \gamma^{-1}) \htp f(\gamma) *
F(x',t) * F(x',t)^{-1} * f(\bar \gamma)^{-1} \\
&\htp F(\gamma(t),t) * F(\bar \gamma(t),t)^{-1}.
\end{align*}
Similarly $g(\lambda) \htp G(\gamma(t),t) * G(\bar \gamma(t),t)^{-1}$, and
so $f(\lambda) \htp g(\lambda)$, which is to say $f_\#(\lambda) =
g_\#(\lambda)$. Since $\lambda$ is a loop at $x$ and $(x,y) \in
\C_n(f,g)$, we have $\sign(\lambda) = \sign(g(\lambda))$ as desired.
\end{proof}

Note that we will not expect an analogue to Lemma \ref{hrel}: if we have a geometrically inessential coincidence class $C$ of $(f',g')$ 
which is related by the homotopies to an empty coincidence class of $(f,g)$, then the points in $C$ can be assigned either orientation 
without violating the $(F,G)$-relatedness criterion. Note that these choices of orientations would not effect the index, because the class is geometrically inessential.

\subsection{Axioms and uniqueness}
 We begin this subsection  defining the tuples for which we will associate an index. 
\begin{defn}
Let $\admm$ be the set of all tuples $(f,g,U,V,\OO)$, where $(f,g,U,V) \in \admm_0$ and $\OO$ is a Nielsen coherent orientation of $\C(f,g,V)$.
\end{defn}

We say that two tuples $(f,g,U,V,\OO), (f',g',U,V,\OO') \in \admm$ are \emph{admissibly homotopic} if there is a pair of admissible (in the sense of Section \ref{otruesection}) homotopies $F,G$ taking $f$ to $f'$ and $g$ to $g'$ and $\OO'$ is $F,G$-related to $\OO$.

Throughout this section we consider functions $\iota:\admm \to \R \oplus \Z_2$. Our main result is that there is a unique such function satisfying certain axioms similar to those used in the previous section. The axioms for this section take the following form:

\begin{axiom}[Additivity axiom]
Given $(f,g,U, V,  \mathbb O)\in \admm$, if $U_1$ and
$U_2$ are disjoint open
subsets of $U$ with $\Coin(f,g,U) \subset U_1 \cup U_2$, 
then
\[ \iota(f,g,U, V, \mathbb O) = \iota(f,g,U_1, V, \mathbb O) +
\iota(f,g,U_2, V, \mathbb O). \]
\end{axiom}

\begin{axiom}[Homotopy axiom]
If $(f, g, U, V, \mathbb O) \in \admm$ and $(f',
g', U, V, \mathbb O')\in \admm$ are
admissably homotopic, then
\[ \iota(f, g, U, V,  \mathbb O) = \iota(f', g',U, V,   \mathbb O'). \]
\end{axiom}

\begin{axiom}[Normalization axiom] 
Let $(c,g,U,V,\OO)\in \admm$ with $c|_U: U \to N$ a constant map $g|_{U}:U \to N$ an embedding with $g(x)=c$ and $\sign(g,\OO(x,c))=1$. Then
\[ \iota(c,g,U,V,\OO) = 
\begin{cases} 
1 \in \Z \quad &\text{if $(x,c) \in \C_n(f,g,M)$}, \\
\bar 1 \in \Z_2 \quad &\text{if $(x,c) \in \C_d(f,g,M)$}. 
\end{cases} \]
\end{axiom}

Exactly as in Section \ref{otruesection} we obtain an excision property.
\begin{prop}[Excision property]
Let $\iota:\adm \to \R$ satisfy the Additivity axiom.
Let $(f,g,U,V,\OO) \in \adm$ and let $U' \subset U$ be an open set with $\Coin(f,g,U) \subset U'$. Then
\[ \iota(f,g,U,V,\OO) = \iota(f,g,U',V,\OO). \]
\end{prop}

The proof of the uniqueness in this section is similar to that in the previous section. The analogue of Lemma \ref{naturality} is that the index at a single nondegenerate coincidence point agrees with $\ind$ (the unique coincidence index for selfmaps of $\R^n$) when we compose with charts.

\begin{lem}\label{nondegnaturality}
Let $(f,g,U,V,\OO) \in \admm$ with $\Coin(f,g,U)$ consisting of a single nondegenerate coincidence point $x$ such that $U$ is a euclidean neighborhood of $x$ and $f(U) \cup g(U)$ is contained in a euclidean neighborhood $W$ of $y=f(x)=g(x)$. Let $j:U \to \R^n$ and $h:W \to \R^n$ be homeomorphisms, and let $\bar f = h \circ f \circ j^{-1}$ and $\bar g = h \circ g \circ j^{-1}$. 

Let $\OO(x,y) = [O_x,O_y]$, and 
define an orientation $\bar \OO$ of $\C(\bar f, \bar g)$ by 
$\bar \OO(j(x),h(y)) = [j_*(O_x), h_*(O_y)]$.
If $\bar \OO$ is oriented consistently, and $\iota$ satisfies the three axioms, then 
\[ \iota(f,g,U,V,\OO) = \ind(\bar f, \bar g, \R^n, \bar \OO). \]
\end{lem}
\begin{proof}
The proof is almost identical to that of Lemma \ref{naturality}. We will highlight the relevant changes.

Let $\mathcal U \subset \admm$ be the set of all tuples $(s,t,W,V,\mathbb A)$ with $W \subset U$ and $\bar {\mathbb A}$ oriented consistently. Then there is a bijection $\omega:\mathcal U \to \adm_{\R^n}$ given by $\omega(s,t,W,V,\mathbb A) = (\bar s, \bar t, j(W), \bar {\mathbb A})$, and so
we have a map $\iota \circ \omega^{-1}: \adm_{\R^n} \to \R$. 

As in the proof of Lemma \ref{naturality}, the proof is complete when we show that $\iota \circ \omega^{-1}$ satisfies the additivity, homotopy, and weak normalization axioms of \cite{stae08a}.
Additivity and homotopy are clear, and the demonstration of the weak normalization axiom from the proof of Lemma \ref{naturality} also applies in this case, though the usage of the Normalization Axiom is only valid if $x$ is nondegenerate. We have assumed this nondegeneracy as a hypothesis, so the proof carries without further changes.
\end{proof}

Our uniqueness proof for this section follows the proof of Theorem \ref{otrueuniqueness}, but requires some extra argument for the degenerate coincidence points. 

\begin{thm}\label{nonotrueuniqueness}
There is at most one function $\iota:\admm \to \R \oplus \Z_2$ which satisfies the additivity, homotopy, and normalization axioms.
\end{thm}
\begin{proof}
Let $(f,g,U,V,\OO) \in \admm$, and let $\iota$ satisfy the axioms. As in the proof of Theorem \ref{otrueuniqueness}, we assume that $\Coin(f,g,U)$ is a set of isolated coincidence points. For each $x\in \Coin(f,g,U)$, let $U_x$ be a euclidean neighborhood of $x$ so that $f(U_x) \cup g(U_x)$ is contained in a euclidean neighborhood $W_x$ of $y=f(x)=g(x)$, and we have
\[ \iota(f,g,U,V,\OO) = \sum_x \iota(f,g,U_x,V,\OO). \]

When $(x,y) \in \C_n(f,g,V)$ the value of $\iota(f,g,U_x,V,\OO)$ is determined uniquely by Lemma \ref{nondegnaturality}. Thus we need only show that $\iota(f,g,U_x,V,\OO)$ is unique for $(x,y) \in \C_d(f,g,V)$. This we do by showing that our pair of maps can be changed by admissable homotopies so that locally they consist of a constant and an embedding. Then the Normalization Axiom gives a unique value for the index. (Actually, the following argument will work whether $(x,y)$ is degenerate or not.)

Let $j:U_x \to \R^n$ and $h:W_x \to \R^n$ be embeddings, and let $\bar f = h \circ f \circ j^{-1}$ and $\bar g = h \circ g \circ j^{-1}$. 
By transversality arguments, we may change $\bar f$ and $\bar g$ to maps $\bar f'$ and $\bar g'$ by admissible homotopies so that: $\bar f'$ and $\bar g'$ are smooth, have finitely many coincidence points, and each coincidence point $z \in \Coin(\bar f', \bar g')$ is regular in the sense that $d\bar g'_z - d\bar f'_z$ (the difference of the derivatives at $x$) is nonsingular. Further, we can choose the homotopies to be sufficiently small so that each coincidence point $z\in \Coin(\bar f', \bar g',U_x)$ is degenerate.

To simplify the argument, we assume that $\Coin(\bar f', \bar g')$ is a single point $z$. (If there are more coincidences we isolate them with neighborhoods and use excision.) By excision and the nonsingularity of the derivatives, we may assume that $U$ is suffiently small so that $\bar g' - \bar f'$ is a homeomorphism

Being a selfmap of $\R^n$, the map $\bar f'$ is
homotopic to the constant map at $0$, and so $\bar g' - \bar f'$ is homotopic to $\bar g'$. In fact this homotopy is admissable: Let $\bar F(z,t) = (1-t)\bar f'(z)$ be the homotopy of $\bar f'$ to $0$, and then the homotopy $\bar G(z,t) = \bar g'(z) - \bar F(z,1-t)$ is a homotopy of $\bar g'$ to $\bar g' - \bar f'$ with $\Coin(\bar F, \bar G, \R^n\times [0,1]) = \{0\} \times I$. This coincidence set is compact, and thus the homotopy is admissable.

The homotopies $h^{-1} (\bar F(j(z), t))$ and $h^{-1}(\bar G(j(z),t))$ give  an
admissible homotopy of $(f,g,U,V,\OO)$ to the tuple $(f',g',U,V,\OO')$,
where $f'|_{U_x}$ is the constant map with constant value $y$ and $g'|_{U_x}$
is the homeomorphism $g'|_{U_x}: U_x \to j^{-1}(\R^n)$ given by 
\[ g'|_{U_x} = h^{-1} \circ (\bar g' - \bar f') \circ j, \]
and $\OO'$ is an orientation of $\C(f',g')$ which is related by the homotopies to $\OO$. By perhaps composing with an orientation reversing automorphism of $\R^n$, we can make our initial choices of $j$ and $h$ so that $\sign(g', \OO(x,y)) = 1$.

By the Homotopy Axiom we have
\[ \iota(f,g,U,V,\OO) = \iota(f',g',U,V,\OO'), \]
and by the Normalization Axiom this value is $\bar 1 \in \Z_2$, and thus is unique.
\end{proof}

As in Corollary \ref{zvalued}, we obtain:
\begin{cor}
If $\iota:\admm \to \R \oplus \Z_2$ satisfies the axioms, it has values in $\Z \oplus \Z_2$.
\end{cor}

\section{The values of the index}\label{valuessection}
Sections \ref{otruesection} and \ref{nonotruesection} discuss coincidence indices with values in $\Z$ and $\Z\oplus \Z_2$, respectively. 
It is natural, especially in the second case, to wonder why these groups in particular are chosen for the values of the index. 
Local indices with
values in other abelian groups arise in the setting of coincidences of  
maps from a complex into
a manifold of the same dimension in \cite{gonc99}, fixed points of  
fiber preserving maps in \cite{dold74}, and coincidences in positive codimension 
in \cite{kosc04}, \cite{kosc06}.

In this section we show that for maps on manifolds, any index-like functions having values in an abelian group must essentially be the same as our index functions above with values in $\Z$ or $\Z \oplus \Z_2$.

First we consider a setting analogous to Section \ref{otruesection}. Then our uniqueness result is as follows:
\begin{thm}\label{otrueGindex}
Let $G$ be an abelian group, and let $a \in G$. Then there is at most one function $\iota_a: \adm \to G$ satisfying the following axioms:
\begin{itemize}
\item(Additivity)
Given $(f,g,U, \mathbb O)\in \adm$, if \ $U_1$ and
$U_2$ are disjoint open
subsets of $U$ with $\Coin(f,g, U) \subset U_1 \cup U_2$, 
then
\[ \iota_a(f,g,U, \mathbb O) = \iota_a(f,g,U_1, \mathbb O) +
\iota_a(f,g,U_2, \mathbb O). \]

\item (Homotopy)
If $(f, g, U, \mathbb O) \in \adm$ and $(f',
g', U, \mathbb O')\in \adm$ are
admissibly homotopic, then
\[ \iota_a(f, g, U, \mathbb O) = \iota_a(f', g',U, \mathbb O'). \]

\item (Normalization)
Let $(c,g,U,\OO) \in \adm$ with $c|_U$ a constant map with constant value $c$ and $g|_U$ an embedding with $g(x) = c$. If $\sign(g,\OO(x,c)) = 1$, then 
\[ \iota_a(c,g,U,\OO) = a. \]
\end{itemize}
\end{thm}

The proof of the above cannot follow exactly the proof of Theorem \ref{otrueuniqueness} without a $G$-valued analogue of Lemma \ref{naturality}. But the ideas from the proof of Theorem \ref{nonotrueuniqueness} can apply to change any pair of maps $(f,g)$ by homotopy so that, near each isolated coincidence point, the pair consists locally of a constant and an embedding. In this case the normalization axiom will determine uniquely the values of $\iota_a$. Since the value of $\iota_a$ can always be computed using the additivity and normalization axioms, we obtain
\begin{cor}
Let $G$ be an abelian group with $a\in G$, and let $\iota_a:\adm \to G$ satisfy the axioms of Theorem \ref{otrueGindex}. Then the values of $\iota_a$ are always in $\langle a \rangle$, the subgroup of $G$ generated by $a$.
\end{cor}

In the case where $a\in G$ has infinite order, the subgroup $\langle a \rangle$ is isomorphic to $\Z$. This is the situation of Section \ref{otruesection}, where $G$ is taken to be $\R$ and $a=1$. Typically it will be most useful in coincidence theory to choose $a$ to have infinite order: if $a$ has some finite order $k$, then it will be possible to have $k$ coincidence points which cannot be removed by homotopies, and yet the value of $\iota_a$ will be $0 \in G$. By the classical theorem of Wecken, this should not be possible when the dimension of the manifolds is not 2. Thus we have:
\begin{thm}
Let $G$ be an abelian group with $a\in G$, and let $\iota_a:\adm \to G$ satisfy the axioms of Theorem \ref{otrueGindex}. Additionally assume that when the domains of $f$ and $g$ are not dimension 2 and $\Coin(f,g,U)$ cannot be made empty by homotopy, the index $\iota_a(f,g,U,\OO)$ is nonzero. 

Then $a$ has infinite order, and so the values of $\iota_a$ are in $\langle a \rangle \cong \Z$.
\end{thm}

Similar arguments could be made for $G$-valued functions on the classes $\adm'$ and $\hat \adm$.

Now we turn to the class $\admm$. We obtain a uniqueness result, proved analogously to the above, along with a weak characterization of the values of the index:
\begin{thm}\label{nonotrueGindex}
Let $G$ be an abelian group, and let $a,b \in G$. Then there is at most one function $\iota_{a,b}: \admm \to G$ satisfying the following axioms:
\begin{itemize}
\item(Additivity)
Given $(f,g,U, V,\mathbb O)\in \admm$, if \ $U_1$ and
$U_2$ are disjoint open
subsets of $U$ with $\Coin(f,g, U) \subset U_1 \cup U_2$, 
then
\[ \iota_{a,b}(f,g,U, V,\mathbb O) = \iota_{a,b}(f,g,U_1, V,\mathbb O) +
\iota_{a,b}(f,g,U_2, V,\mathbb O). \]

\item (Homotopy)
If $(f, g, U, V,\mathbb O) \in \admm$ and $(f',
g', U, V,\mathbb O')\in \adm$ are
admissibly homotopic, then
\[ \iota_{a,b}(f, g, U,V, \mathbb O) = \iota_{a,b}(f', g',U, V,\mathbb O'). \]

\item (Normalization)
Let $(c,g,U,V,\OO) \in \admm$ with $c|_U$ a constant map with constant value $c$ and $g|_U$ an embedding with $g(x) = c$. If $\sign(g,\OO(x,c)) = 1$, then 
\[ \iota_{a,b}(c,g,U,V,\OO) = \begin{cases} 
a \quad &\text{if $(x,c) \in \C_n(f,g,M)$} \\
b \quad &\text{if $(x,c) \in \C_d(f,g,M)$} \end{cases}
\]
\end{itemize}
Furthermore, the values of such a function $\iota_{a,b}$ are always in $\langle a, b\rangle$, the subgroup of $G$ generated by $a$ and $b$.
\end{thm}

We can be more specific about the values of the index above, which must be in $\langle a, b \rangle$. Note that when the domains of $f$ and $g$ are nonorientable, it is possible to construct maps having two degenerate coincidence points which can be combined and removed by homotopy. The sum index of these coincidence points must be $2b$, but this index must also be zero since the coincidences can be removed by homotopy. Thus we have $2b = 0$ in $G$, which means that $b$ has order (at most) 2. Thus similarly to the above we have
\begin{thm}
Let $G$ be an abelian group with $a,b\in G$, and let $\iota_{a,b}:\admm \to G$ satisfy the axioms of Theorem \ref{nonotrueGindex}. Additionally assume that when the domains of $f$ and $g$ are not dimension 2 and $\Coin(f,g,U)$ cannot be made empty by homotopy, the index $\iota_{a,b}(f,g,U,V,\OO)$ is nonzero. 

Then $a$ has infinite order and $b$ has order 2, and so the values of $\iota_{a,b}$ are in $\langle a,b \rangle \cong \Z \oplus \Z_2$.
\end{thm}

\section{The local Reidemeister trace}\label{rtsection}
Given local maps $f, g$ of $M\to N$ with connected domain $V\subset M$,
we partition the group $\pi_1(N)$ into \emph{Reidemeister
     classes} as follows: two elements $\alpha, \beta \in \pi_1(N)$ are
equivalent if and only if there is some $\gamma \in \pi_1(V)$ with
\[ \alpha = g_\#(\gamma) \beta f_\#(\gamma)^{-1}, \]
where $f_\#,g_\#: \pi_1(V) \to \pi_1(N)$ are the homomorphisms
induced by $f$ and $g$. Let $\Reid(f,g)$ denote the set of
Reidemeister classes defined by $f$ and $g$.

If we fix lifts (local maps) $\lift f,\lift g$ of $\lift M \to \lift N$ to the universal covering space, we can assign a Reidemeister class to each coincidence point. Let $x\in \Coin(f,g)$ with a point $\lift x \in p^{-1}(x)$, where $p:\lift M \to M$ is the covering space projection. Then there will be some $\alpha \in \pi_1(N)$ (now viewing $\alpha$ as a covering transformation) with $\lift x \in \Coin(\alpha \lift f, \lift g)$. This element gives the class $[\alpha] \in \Reid(f,g)$, which we call the Reidemeister class of $x$. The correspondence of coincidence points and Reidemeister classes is well defined with respect to coincidence classes: two points in the same coincidence class will have the same Reidemeister class.

The Reidemeister trace (sometimes called generalized Lefschetz number) is an invariant which captures data concerning both the coincidence index
and the Reidemeister classes of coincidence points.  
In this section we show that the ideas from the previous sections
extend to a uniqueness result for the local Reidemeister trace following the
structure of \cite{stae08a}. The main result (Theorem 3) of that paper
made no explicit use of the orientability or differentiability
hypotheses, using only the uniqueness of the coincidence index.

\subsection{The case where the second map is orientation true}
Let $\lift \adm$ be the set of all tuples $(f,\lift f, g, \lift g, U, V, N, \OO)$
where $f,g:V \subset M \to N$ with $U\subset V$, and $(f,g,U,\OO) \in \adm$ (so $g$ is orientation true), the set $V$ is the connected domain of $f$ and $g$, and $\lift f, \lift g$ (local maps) of $\lift M \to \lift
N$ are lifts of $f$ and $g$ to the universal covering spaces. We
say that two such tuples $(f,\lift f, g, \lift g, U,V,N,\OO)$ and $(f', \lift
f', g', \lift g', U,V,N,\OO')$ are admissably homotopic if there is an
admissible (in the sense of Section \ref{otruesection}) pair of homotopies $F,G$ of $(f,g,U,N,\OO)$ to $(f',g',U,N,\OO')$ 
which lifts to a pair of homotopies of $\lift f$ to $\lift f'$ and
$\lift g$ to $\lift g'$.
The Reidemeister trace is a function $\RT$ defined on $\lift \adm$
which assigns to a tuple $(f, \lift f, g, \lift g, U,V,N,\OO)$ an element of
the abelian  group $\Z\Reid(f,g)$. 
We view this group as the set of finite formal sums of elements of $\Reid(f,g)$ with integer coefficients. Let $\epsilon: \Z\Reid(f,g) \to \Z$ be the sum of the coefficients.

Note that in this section we must keep track of $V$ and $N$ in our tuples. If we have $f,g:M \to N$ and $N$ is an open submanifold of some $N'$, then it is possible for $\pi_1(N) \neq \pi_1(N')$ (and similarly for $\pi_1(V) \neq \pi_1(V')$ if we change the domain $V$), and so the set $\Reid(f,g)$ will depend on the precise choice of the domain and codomain of $f$.

Our uniqueness theorem is as follows:
\begin{thm}\label{otruertuniqueness}
There is at most one function $\RT$ defined on $\lift \adm$
which assigns to a tuple $(f, \lift f, g, \lift g, U,V,N,\OO)$ an element of
the Abelian group $\Z\Reid(f,g)$ and satisfies the following axioms:
\begin{itemize}
\item (Additivity) Given $(f, \lift f, g,
     \lift g, U,V,N,\OO) \in \lift \adm$, if $U_1$ and $U_2$ are disjoint open subsets
     of $U$ with $\Coin(f,g,U) \subset U_1 \cup U_2$, then
\[ \RT(f,\lift f, g, \lift g, U,V,N,\OO) = \RT(f, \lift f,  g, \lift g, U_1,V,N,\OO) +
     \RT(f, \lift f, g, \lift g, U_2,V,N,\OO). \]

\item (Homotopy) If $(f,\lift g, \lift g, U,V,N,\OO), (f',\lift f', g',\lift
     g', U,V,N,\OO') \in \lift \adm$ are admissably homotopic,  then
\[ \RT(f, \lift f, g, \lift g, U,V,N,\OO) = \RT(f', \lift f', g', \lift g',
U,V,N,\OO'). \]

\item (Normalization)
Let $(c,\lift c,g,\lift g,U,V,N,\OO) \in \lift \adm$ with $c|_U$ a constant map and $g|_U$ an embedding with $g(x) = c$. If $\sign(g,\OO(x,c)) = 1$, then 
\[ \epsilon \circ \RT(c,\lift c,g,\lift g,U,V,N,\OO) = 1. \]

\item (Coefficients invariance)
For any $(f, \lift f, g, \lift
g, U,N,\OO) \in \lift \adm$, and any $\alpha, \beta \in \pi_1(N)$ and any manifold $N'$ which contains $N$ as an open submanifold, we have
\[ \epsilon(\RT(f, \lift f, g, \lift g, U,V,N,\OO)) = \epsilon(\RT(f, \alpha \lift f, g,
\beta \lift g, U,V,N',\OO)). \]

\item (Coincidence of lifts) If $[\alpha] \in \Reid(f,g)$ appears with nonzero
     coefficient in $\RT(f,\lift f, g, \lift g, U,V,N,\OO)$, then $\alpha \lift
     f$ and $\lift g$ have a coincidence on $p^{-1}(U)$, where $p:
     \lift M \to M$ is the covering space projection.
\end{itemize}
\end{thm}

The proof is analogous to that found in \cite{stae08a}, we will give a sketch:
Letting $\RT$ satisfy the above axioms, the composition $\epsilon \circ \RT: \lift \adm \to \Z$ is independent of the lifts $\lift f$ and $\lift g$ and the codomain $N$ by the coefficients invariance axiom, and so we may regard $\epsilon \circ \RT$ as a function 
on $\adm$. This function satisfies the corresponding axioms of Section \ref{otruesection}, and so by Theorem \ref{otrueuniqueness}, $\epsilon \circ \RT$ is uniquely determined.

Now to show that $\RT$ is unique, we use the Homotopy axiom, and Additivity and excision to form a sum over euclidean neighborhoods containing one coincidence point each. On some such neighborhood, say a neighborhood $U_x$ of a coincidence point $x$, the value of $\RT$ must be some coefficient times the Reidemeister class associated to $x$ (the appearance of any other Reidemeister class would contradict the Coincidence of lifts axiom). By the above paragraph, though, the coefficient must be given by the unique function $\iota$. Thus we have 
\[ \RT(f,\lift f, g, \lift g, U, V,N, \OO) = \sum_{x\in \Coin(f,g,U)} \iota(f,g,U_x,\OO) [x], \]
where $[x]\in \Reid(f,g)$ is the Reidemeister class associated to $x$. The above computation of $\RT$ is derived using only the axioms, and thus establishes its uniqueness.

Note that appropriate versions of Corollaries \ref{swap}, \ref{localization}, and \ref{flipsign} will hold for this Reidemeister trace, by the same arguments. Also, the results of this subsection could be adapted as in Subsection \ref{forgsubsection} to give a unique Reidemeister trace for tuples where the first map is orientation true, or tuples where at least one map is orientation true.

\subsection{The general case}
Now we turn to the case where $g$ is not assumed to be orientation true. Let $\lift \admm$ be the set of tuples $(f,\lift f, g, \lift g, U, V, \OO)$ where $(f,g,U,V,N,\OO) \in \admm$, the manifold $N$ is the codomain of $f$ and $g$, and $\lift f, \lift g$ are lifts of $f$ and $g$. Define admissable homotopy of tuples as above.

In this general case (without any orientation-true assumptions), $\RT(f,\lift f, g, \lift g, U, V, N, \OO)$ has value in the abelian group $(\Z\oplus\Z_2)\Reid(f,g)$. Again, let $\epsilon: (\Z \oplus \Z_2)\Reid(f,g) \to \Z \oplus \Z_2$ be the sum of the coefficients. Our uniqueness result follows from exactly the same arguments as in Theorem \ref{otruertuniqueness}. 

\begin{thm}
There is a unique function $\RT$ defined on $\lift \admm$
which assigns to a tuple $(f, \lift f, g, \lift g, U,V,N,\OO)$ an element of
the Abelian group $(\Z\oplus \Z_2)\Reid(f,g)$ and satisfies the following axioms:
\begin{itemize}
\item (Additivity) Given $(f, \lift f, g,
     \lift g, U,V,N,\OO) \in \lift \admm$, if $U_1$ and $U_2$ are disjoint open subsets
     of $U$ with $\Coin(f,g,U) \subset U_1 \cup U_2$, then
\[ \RT(f,\lift f, g, \lift g, U,V,N,\OO) = \RT(f, \lift f,  g, \lift g, U_1,V,N,\OO) +
     \RT(f, \lift f, g, \lift g, U_2,V,N,\OO). \]

\item (Homotopy) If $(f,\lift g, \lift g, U,V,N,\OO), (f',\lift f', g',\lift
     g', U,V,N,\OO') \in \lift \admm$ are admissably homotopic,  then
\[ \RT(f, \lift f, g, \lift g, U,V,N,\OO) = \RT(f', \lift f', g', \lift g',
U,V,N,\OO'). \]

\item (Normalization)
Let $(c,\lift c,g,\lift g,U,V,N,\OO) \in \lift \admm$ with $c|_U$ a constant map and $g|_U$ an embedding with $g(x) = c$. If $\sign(g,\OO(x,c)) = 1$, then 
\[ \epsilon \circ \RT(c,\lift c,g,\lift g,U,V,N,\OO) = 
\begin{cases} 
1 \in \Z \quad &\text{if $(x,c) \in \C_n(c,g,V)$}, \\
\bar 1 \in \Z_2 \quad &\text{if $(x,c) \in \C_d(c,g,V)$}. 
\end{cases} \]

\item (Coefficients invariance)
For any $(f, \lift f, g, \lift
g, U,V,N,\OO) \in \lift \admm$, and any $\alpha, \beta \in \pi_1(N)$ and any manifold $N'$ which contains $N$ as an open submanifold, we have
\[ \epsilon(\RT(f, \lift f, g, \lift g, U,V,N,\OO)) = \epsilon(\RT(f, \alpha \lift f, g,
\beta \lift g, U,V,N',\OO)). \]

\item (Coincidence of lifts) If $[\alpha] \in \Reid(f,g)$ appears with nonzero
     coefficient in $\RT(f,\lift f, g, \lift g, U,V,N,\OO)$, then $\alpha \lift
     f$ and $\lift g$ have a coincidence on $p^{-1}(U)$, where $p:
     \lift V \to V$ is the covering space projection.
\end{itemize}
\end{thm}
We end the section with a note about the value of the Reidemeister trace. We can be a bit more specific about where this value lies by partitioning the Reidemeister classes into degenerate and nondegenerate classes following \cite{gk91}, Section 2, Definition 1.

\begin{defn} 
Let $(f,g, U,V,N,\OO) \in \admm$. A Reidemeister class $[\alpha]$ is \emph{degenerate} if there exists a loop $\gamma \in \pi_1(V)$ 
 such that $\alpha=g_{\#}(\gamma)\alpha(f_{\#}(\gamma)^{-1})$ and $\sign(\gamma) \neq \sign g_{\#}(\gamma)$ for some $\alpha\in[\alpha]$.
Otherwise we say that $[\alpha]$ is \emph{nondegenerate}.

Denote by $\Reid_d(f,g,V)$ the set of degenerate Reidemeister classes and by
$\Reid_n(f,g,V)$ the set of nondegenerate Reidemeister classes. 
\end{defn}

Our final result is that, in the Reidemeister trace, the nondegenerate Reidemeister classes always appear with integer coefficients, while the degenerate classes always appear with $\Z_2$ coefficients. Since the integer coefficients and mod 2 coefficients appear respectively for nondegenerate and degenerate coincidence points, this essentially means that the degeneracy of points corresponds in the appropriate way with degeneracy of Reidemeister classes.

\begin{thm}
For $(f,\lift f, g, \lift g, U,V,N,\OO) \in \lift \admm$, we have 
\[ \RT(f,\lift f,g,\lift g, U, V, N,\OO) \in \Z\Reid_n(f,g,V) \oplus \Z_2 \Reid_d(f,g,V). \]
\end{thm}
\begin{proof}
By the Homotopy and Additivity axioms we may replace $f$ and $g$ by a pair with isolated coincidences and we have
\[ \RT(f,\lift f, g, \lift g, U, V, N, \OO) = \sum_{x\in \Coin(f,g,U)} \iota(f,g,U_x,V,N,\OO) [x], \]
where $U_x$ is a euclidean neighborhood of $x$ containing no other coincidences, and $[x]\in \Reid(f,g)$ is the Reidemeister class associated to $x$.

By the proof of our uniqueness for $\RT$, the coefficient on $[x]$ must be the coincidence index $\iota$ on $U_x$. By the proof of Theorem \ref{nonotrueuniqueness}, this coefficient is in $\Z$ when $x \in \C_n(f,g,V)$, and is in $\Z_2$ when $x\in \C_d(f,g,V)$. (The maps can be changed by admissable homotopy so that each coincidence point has a neighborhood on which the Normalization Axiom applies.)

Thus it suffices to show that $x \in \C_n(f,g,V)$ if and only if $[x] \in \Reid_n(f,g,V)$, which will imply that $x \in \C_d(f,g,V)$ if and only if $[x]\in \Reid_d(f,g,V)$.

The class $[x]$ is the Riedemeister class of the covering transformation $\alpha$ such that $x \in p \Coin(\alpha \lift f, \lift g)$. We begin by showing that if $x \in \C_n(f,g,V)$, then $[x] \in \Reid_n(f,g,V)$. 

Let $[\alpha]=[x]$ be the Reidemeister class of $x$. This means that there is a point $\lift x$ with $p(\lift x) = x$ and $\alpha \lift f(\lift x) = \lift g (\lift x)$. To show that $[\alpha] \in \Reid_n(f,g,M)$, we assume that $\alpha = g_\#(\gamma) \alpha f_\#(\gamma^{-1})$ for some $\gamma \in \pi_1(V,x_0)$, and we will show that $\sign \gamma = \sign g_\#(\gamma)$.

Let $\lambda:[0,1] \to V$ be a path from $x_0$ to $x$.
Let $\theta$ be the loop at $x$ given by $\theta = \lambda * \gamma * \lambda^{-1}$, and let $\lift \theta$ be the lift of $\theta$ with initial point $\lift x$. The terminal point of $\lift \theta$ will be $\gamma \lift x$, viewing $\gamma$ as a covering transformation. 

Then the terminal point of $\alpha \lift f(\lift \theta)$ will be 
\[ \alpha \lift f(\gamma \lift x) = \alpha f_\#(\gamma) \lift f(\lift x) = g_\#(\gamma) \alpha \lift f (\lift x) = g_\#(\gamma) \lift g(\lift x) = \lift g(\gamma \lift x) , \]
which is the terminal point of $\lift g(\lift \theta)$. These paths $\alpha \lift f(\lift \theta)$ and $\lift g(\lift \theta)$, being two paths in the universal cover with the same endpoints, are homotopic. Thus their projections are homotopic, and so $f(\theta) \htp g(\theta)$.

Since $x \in \C_n(f,g,V)$ and $\theta$ is a path at $x$ with $f(\theta) \htp g(\theta)$, we have $\sign \theta = \sign g(\theta)$. Since $\theta = \lambda * \gamma * \lambda^{-1}$, this means that $\sign \gamma = \sign g_\#(\gamma)$ as desired.

Now we assume that $[\alpha] = [x] \in \Reid_n(f,g,V)$, and we will show that $x \in \C_n(f,g,V)$. Let $\theta$ be a loop at $x$ with $f(\theta) \htp g(\theta)$, and we will show that $\sign \theta = \sign g(\theta)$.

As above, choose $\lift x$ with $p(\lift x) = x$ and $\alpha \lift f(\lift x) = \lift g (\lift x)$, and let $\lambda$ be a path from $x_0$ to $x$. Let 
$\gamma = \lambda^{-1} * \theta * \lambda$. Let $\lift \theta$ be the lift of $\theta$ with initial point $\lift x$. Then the terminal point of $\lift \theta$ will be $\gamma \lift x$.

Since $f(\theta) \htp g(\theta)$, the paths $\alpha \lift f(\lift \theta)$ and $\lift g(\lift \theta)$ will be homotopic, and in particular have the same endpoints. The terminal point of the former is $\alpha\lift f(\gamma \lift x)$, while the terminal point of the latter is $\lift g(\gamma \lift x)$. Equating these gives $\alpha f_\#(\gamma) \lift f(\lift x) = g_\# (\gamma) \lift g(\lift x)$, and since $\lift g(\lift x) = \alpha \lift f(\lift x)$ we have $\alpha f_\#(\gamma) \lift f(\lift x) = g_\#(\gamma) \alpha \lift f(\lift x)$, which means $\alpha f_\#(\gamma) = g_\#(\gamma) \alpha$, and thus $\alpha = g_\#(\gamma) \alpha f_\#(\gamma^{-1})$.

Since $[\alpha] \in \Reid_n(f,g,V)$ we have $\sign \gamma = \sign g_\#(\gamma)$. Since $\gamma = \lambda^{-1} * \theta * \lambda$, this means that $\sign \theta = \sign g(\theta)$ as desired.
\end{proof}

As a final note, we point out that the ideas of Section \ref{valuessection} could be extended using the same arguments to specify the values of the Reidemeister trace. Specifically we can show that any appropriate function on $\lift \adm$ with values in $G\Reid(f,g)$ for some abelian group $G$ must always have values in $\langle a \rangle \Reid(f,g) \cong \Z\Reid(f,g)$ for some infinite order element $a\in G$. Similarly, any appropriate function on $\lift \admm$ with values in $G \Reid(f,g)$ must always have values in $\langle a \rangle \Reid_n(f,g,V) \oplus \langle b \rangle \Reid_d(f,g,V)$ where $a\in G$ has infinite order and $b\in G$ has order 2.

\section*{Appendix- Existence}
Here we briefly discuss the existence of the various indices and Reidemeister traces above. 

For the coincidence index $\iota:\adm \to \Z$ of Subsection \ref{otruesection}, the proof of Theorem \ref{otrueuniqueness} gives a construction: choose small Euclidean neighborhoods around each coincidence point (assuming there are finitely many), and use the classical local coincidence index on $\R^n$ after composing through charts. It could be checked that $\iota$ defined in such a way does not depend on the choices of charts, and that it does indeed satisfy the three axioms of Subsection \ref{otruesection}. (These checks, as well as the assumption that the coincidence set be finite, are avoided using a homological construction outlined below.)

A similar strategy suffices to construct the unique function $\iota:\admm \to \Z \oplus \Z_2$.
Let $f,g$ be local maps of $M \to N$, and as before assume that the coincidence set is finite. For each $x\in \Coin(f,g, U)$, we can compute the above classical local index on a euclidean neighborhood of $x$. This computation  
yields an integer whose sign depends on choices of orientations made  
at the point $x$ and the point $f(x)$.
If $U_x \subset U$ is a Euclidean neighborhood of $x$ which contains  
only $x$ as a coincidence point,
this classical index will coincide with what is denoted in our  
notation by $\iota(f,g,U_x,U_x,N,\OO)$, where $\OO$ agrees with the  
orientations chosen above. Note that $\iota(f,g,U_x,U_x,N,\OO)$ is an  
integer (not in $\Z_2$), since $U_x$ is orientable.

The unique function $\iota:\admm \to \Z \oplus \Z_2$ discussed in this paper can be constructed
in terms of the classical indices by the formula
\begin{equation}\label{psum}
\iota(f,g,U,V,N,\OO) =  \left( \sum_{(x,y)\in \C_n(f,g,U)}  
\iota(f,g,U_x,U_x,N,\OO) \right) \oplus q\left( \sum_{(x,y)\in  
\C_d(f,g,U)} \iota(f,g,U_x,U_x,N,\OO) \right),
\end{equation}
where $q:\Z \to \Z_2$ is reduction mod 2.
The procedure above is perhaps the most naive way to define the local  
concepts in question, at least if the number of coincidence points
is finite.

With this approach, it is also straightforward to define a local  
Reidemeister trace in $\Z\Reid_n(f,g,M) \oplus \Z_2\Reid_d(f,g,M)$  
whose image under $\epsilon$ (the sum of the coefficients) is the local coincidence index $\iota$: if the Reidmeister class of  
some coincidence point $x$ is nondegenerate, then the local  
Reidemeister trace is the element of $\Z\Reid_n(f,g,M)$ given by the  
classical index of $x$ times the Reidmeister class of $x$. If this  
class is degenerate, then the local Reidemeister trace is the element  
of $\Z_2\Reid_d(f,g,M)$ given by the mod 2 reduction of the classical  
index of $x$ times the Reidmeister class of $x$.

An alternative approach for these constructions, which will apply in  
the  general case where the set of coincidence points is not finite,  
is via the  homological machinery of \cite{fahuII81}. Certainly
one could deform the maps to make the number of coincidence points  
finite and then define the index in terms of the deformed maps. But  
one would
need to show that the  final result is independent of the deformation.  
This is avoided in the homological approach. 

Given a tuple $(f,\lift f,g,\lift g,U,V,N,\OO) \in \lift \admm$, the obstruction to deform the pair  
to be coincidence free lies (see Theorem 2.9 in  \cite{fahuII81}) in  
the group $H^n_c(V, \mathbb  Z[\pi])$ (cohomology with compact  
support), with $\pi = \pi_1(N)$.
By  Poincar\'e duality this group
  is isomorphic to  the quotient of  $\Z[\pi]$  by the action of  
$\pi_1(V)$ on $\mathrm{Aut}(\Z[\pi])$.
Using this action, which is given in \cite{fahuII81}, it is not  
difficult to show that this quotient is isomorphic to
$\Z\Reid_n(f,g,V) \oplus \Z_2 \Reid_d(f,g,V)$.

Now given $C\subset U$ an isolated set of coincidence points,
we consider the obstruction to deform $f,g:U \to N$ to be coincidence  
free. This is an element $\gamma(f,g)$ of the cohomology $H^n(U, U-C;  
\Z[\pi])$, which
by excision is isomorphic to  $H^n(M, M-C; \Z[\pi])$. Define the class  
$\gamma(f,g; U)\in H^n_c(M, \mathbb \Z[\pi])$ using this isomorphism
and the induced map by the inclusion  $i^*: H^n(M, M-C;   \Z[\pi]) \to  
H^n(M; \Z[\pi])$, which is the local Reidemeister trace. Then the local index  
is the image of  $\gamma(f,g; U)$ in $\mathbb Z \oplus\mathbb Z_2 $ by  
$\epsilon$.

The unique local index $\iota:\admm \to \Z \oplus \Z_2$ discussed in this paper is also closely related to the ``semi-index'' of Dobrenko and Jezierski  
described in \cite{dj93}. We conclude by showing how formula \eqref{psum} for $\iota$ relates to the semi-index.

Recall that when $g$ is orientation true, as we saw in Proposition  
\ref{twoorientations}, a choice of  
orientation at a single coincidence point uniquely determines a  
$g$-coherent orientation of $\C(f,g)$. Thus there are exactly two possible  
$g$-coherent orientations: call them $\OO$ and $-\OO$. From Corollary  
\ref{flipsign}, we have
\[ \iota(f,g,U,N,-\OO) = -\iota(f,g,U,N,\OO). \]

Thus the index in the orientation true case is, up to sign,  
independent of the orientation chosen. Equivalently, we could say that  
the absolute value of the index is independent of the orientation  
chosen.

In the case where $g$ is not orientation true the situation is more  
complicated. We have defined the index in terms of a Nielsen coherent  
orientation, which is uniquely determined when we choose an  
orientation at one point in each coincidence class.

Thus there are perhaps many possible Nielsen coherent orientations,  
and we only expect the index to be independent of orientation (up to  
sign) within Nielsen classes. More precisely, if $\Coin(f,g,U)$ is a  
coincidence class, then there are only two Nielsen coherent  
orientations for $\C(f,g) \cap U \times N$, and they will give opposite values for  
the index.
Thus when $\Coin(f,g,U)$ is a coincidence class, the quantity  
$|\iota(f,g,U,V,\OO)| \in \Z$ will be independent of $\OO$. Here,  
the ``absolute value'' used is
$| \cdot |: \Z \oplus \Z_2 \to \Z$, computed as the  
usual absolute value of the $\Z$ part, plus 1 if the $\Z_2$ part is nontrivial.

This quantity $|\iota(f,g,U,V,N,\OO)|$ (using the above absolute value) coincides  
with the semi-index of \cite{dj93} for the coincidence class  
$\Coin(f,g,U)$. Dobrenko and Jezierski begin by dividing the  
coincidence set (of a regular pair) into ``free elements'' and pairs  
of ``reducing elements'', and define the semi-index as the number of  
coincidence points remaining after deleting a maximal set of reducing  
pairs.

Those coincidence points which appear in reducing pairs are precisely  
the points of $\C_d(f,g,V)$ together with pairs of points of  
$\C_n(f,g,V)$ having opposite indices (see Theorem 5.3 of  
\cite{gj97}). Thus the reducing pairs consist of pairs of  
nondegenerate points whose index sum is zero, plus some even number of  
degenerate coincidence points. The value of the semi-index, then, is  
the (absolute value of the) index sum of the nondegenerate points,  
plus one if the number of degenerate points is odd. This (in absolute  
value) gives precisely \eqref{psum} when $\Coin(f,g,U)$ is a  
coincidence class.

\end{document}